\documentclass[journal]{IEEEtran}
\usepackage{cite}
\usepackage{amsmath}
\usepackage{graphicx}
\usepackage{amsmath}
\usepackage{amssymb}
\usepackage{color}
\usepackage{mathtools}
\usepackage{amsthm}
\usepackage{comment}
\usepackage{esdiff}
\usepackage{cancel}
\usepackage{tikz}
\usetikzlibrary{positioning}
\usetikzlibrary{calc}

\tikzset{
base/.style = {rectangle, draw=black, minimum width = 1cm, minimum height = 0.8cm, text centered},  
arrow/.style = {>=latex, thick}}

\tikzset{
add/.style n args={4}{
    minimum width=6mm,
    path picture={
        \draw[black] 
            (path picture bounding box.south east) -- (path picture bounding box.north west)
            (path picture bounding box.south west) -- (path picture bounding box.north east);
        \node at ($(path picture bounding box.north)+(0,0.19)$)     {\tiny #1};
        \node at ($(path picture bounding box.north)+(0.19,0)$)      {\tiny #2};
        \node at ($(path picture bounding box.north)+(0,-0.19)$)        {\tiny #3};
        \node at ($(path picture bounding box.north)+(-0.19,0)$)     {\tiny #4};
        }
    }
}

\newcommand{\col}{black}
\newcommand{\p}{\mathcal{P}}

\newcommand{\F}{\mathcal{F}}
\newcommand{\G}{\mathcal{G}}

\newcommand{\A}{\mathcal{A}}

\newcommand{\U}{\mathcal{U}}
\newcommand{\X}{\mathcal{X}}
\newcommand{\hp}{{\color{\col}\hat{P}}}

\newcommand{\hpdrv}{{\color{\col}\hat{P}_{drv}}}
\newcommand{\hpdrvk}{{\color{\col}\hat{P}_{drv,k}}}

\newcommand{\hpeng}{{\color{\col}\hat{P}_{eng}}}
\newcommand{\hpengk}{{\color{\col}\hat{P}_{eng,k}}}

\newcommand{\upengk}{{\color{\col}\underline{P}_{eng,k}}}

\newcommand{\opengk}{{\color{\col}\overline{P}_{eng,k}}}

\newcommand{\hpemk}{\hat{P}_{em,k}}

\newcommand{\upemk}{{\color{\col}\underline{P}_{em,k}}}

\newcommand{\opemk}{{\color{\col}\overline{P}_{em,k}}}

\newcommand{\hpb}{{\color{\col}\hat{P}_{b}}}
\newcommand{\hpbk}{\hat{P}_{b,k}}
\newcommand{\upb}{{\color{\col}\underline{P}_{b}}}
\newcommand{\upbk}{{\color{\col}\underline{P}_{b,k}}}
\newcommand{\opb}{{\color{\col}\overline{P}_{b}}}
\newcommand{\opbk}{{\color{\col}\overline{P}_{b,k}}}

\newcommand{\hsk}{{\color{\col}\hat{\sigma}(k)}}

\newcommand{\hE}{{\color{\col}\hat{E}}}
\newcommand{\hEk}{{\color{\col}\hat{E}_k}}
\newcommand{\hf}{{\color{\col}\hat{f}}}
\newcommand{\hfk}{{\color{\col}\hat{f}_k}}
\newcommand{\hh}{{\color{\col}\hat{h}}}
\newcommand{\hhk}{{\color{\col}\hh_k}}
\newcommand{\hg}{{\color{\col}\hat{g}}}
\newcommand{\hgk}{{\color{\col}\hg_k}}
\newcommand{\hgkinv}{{\color{\col}\hat{g}_k^{-1}}}
\newcommand{\ak}[1]{{\color{\col}\alpha_{#1,k}}}
\newcommand{\bk}[1]{{\color{\col}\beta_{#1,k}}}

\def\uu{\underline{u}}
\def\ou{\overline{u}}

\def\ox{\overline{x}}
\def\ux{\underline{x}}

\def\hf{\hat{f}}
\def\hg{\hat{g}}

\def\con{\circ}

\newtheorem{proposition}{Proposition}[section]

\usepackage{algorithm}
\usepackage{algorithmic}

\begin{document}

\title{Energy Management in Plug-in Hybrid Electric Vehicles: Convex Optimization Algorithms for Model Predictive Control}
%
%

\author{Sebastian East, Mark Cannon
\thanks{S. East and M. Cannon are with the Department
of Engineering Science, University of Oxford, Oxford, OX1 3PJ (e-mail: sebastian.east@eng.ox.ac.uk; mark.cannon@eng.ox.ac.uk)}}


\maketitle

\begin{abstract}
 This paper details an investigation into the computational performance of algorithms used for solving a convex formulation of the optimization problem associated with model predictive control for energy management in hybrid electric vehicles with nonlinear losses. A projected interior point method is proposed, where the size and complexity of the Newton step matrix inversion is reduced by applying inequality constraints on the control input as a projection, and its properties are demonstrated through simulation in comparison with an alternating direction method of multipliers (ADMM) algorithm, and general purpose convex optimization software CVX. It is found that the ADMM algorithm has favourable properties when a solution with modest accuracy is required, whereas the projected interior point method is favourable when high accuracy is required, and that both are significantly faster than CVX.
\end{abstract}


\section{Introduction}
\IEEEPARstart{I}{ncreased} electricification of road vehicles has been identified as a key short term solution to important societal issues including climate change and air pollution \cite{EhsaniTextbook}. Plug-in hybrid electic vehicles (PHEVs), where an electric propulsion system is complemented with an internal combustion engine, are currently a common configuration. Although the low energy density and lengthy recharge time of lithium ion batteries limits the viability of all-electric powertrains, analysis of daily driving behaviour reveals that 50\% of internal combustion powered miles can be powered electrically using a hybrid vehicle with an all-electric range of just 40 miles \cite{Ehsani2010}. The inclusion of an additional power source, however, introduces a challenging problem: at each instant during a given journey, how much power should be delivered from the motor, and how much should be delivered from the engine? 

This is known as the energy management problem \cite{Sciarretta2007}, and a simple heuristic is a charge depleting/charge sustaining strategy, where power is delivered from only the electric motor until the battery is sufficiently depleted, and then the vehicle is operated in a charge sustaining mode until the end of the journey \cite{Ehsani2010}. It has, however, been demonstrated that significant savings in fuel consumption can be made by delivering power from both the motor and engine simultaneously, and modulating the fraction delivered from each throughout the journey in what is known as a `blended mode' \cite{MarinaMartinez2016}. There are several methods for controlling the powertrain in this way, and the globally optimal solution can be obtained for complex, nonlinear models and/or integer conrol decisions (such as gear selection) using Dynamic Programming \cite{Sundstrom2010,Larsson2015}, but this approach is too computationally demanding for an online solution. Other studies have investigated methods based on Pontryagin's Minimum Principle \cite{Serrao2013a,Kim2011,Schmid2018}, although it is challenging to enforce complex constraints, such as the general state of charge constraint or engine switching, whilst still guaranteeing optimality.

Model predictive control (MPC) has shown promise in this application due to the inherent robustness to uncertainty in both the vehicle model and prediction of future driving behaviour \cite{Huang2017}, and nonlinear models of losses can be used throughout the hybrid powertrain for improved performance \cite{Buerger}. The associated online optimization problem is still computationally intractable when gear selection and engine switching are considered, so these elements are commonly removed from the problem or optimized externally \cite{Buerger,Elbert2014,Hadj-Said2016,Nuesch2014}, and it has been demonstrated that the power balance alone can be formulated, with nonlinear losses and without simplification, as a convex optimization problem with linear state dynamics \cite{EastCDC2018}.

For the last 30 years, the most popular algorithms for solving inequality constrained optimization problems have been interior point methods \cite{Gondzio2012}. Originally formulated as `primal' methods by approximating inequality constraints in an optimization problem with logarithmic barrier functions, their inherent ill-conditioning and numerical inefficiency rendered interior point methods ineffective until the publication of Karmakar's method \cite{Karmarkar1984} in 1984. The subsequently developed `primal-dual' interior point methods \cite{Potra2000} displayed excellent theoretical and practical properties, including polynomial complexity and a near constant number of iterations with variations in problem size, and today, a large volume of research output in the field of optimization for MPC is dedicated to the development and application of primal-dual interior point methods \cite{Wang2017,Zhang2018,Klintberg2017a}.

In the recently published literature, algorithms for solving convex formulations of the energy management problem have not been investigated and the optimization problem has normally been solved using general purpose convex optimization software \cite{Hadj-Said2016,Egardt2014,Hu2013,Nuesch2014}, with the exception of the alternating direction method of multipliers (ADMM) algorithm presented in \cite{EastCDC2018}. The first contribution of this paper is a projected interior point solver for this problem, where the size of the matrix inversion associated with the primal-dual Newton step is reduced by enforcing the element-wise inequality constraints on the decision variable as a projection, thereby reducing the computational requirement of each iteration. The algorithm is not domain specific, and is applicable to any MPC optimization problem with linear dynamics, a separable convex cost function of the control variable, and upper and lower bounds on the control and state variables. 

A primary motivation of this paper was to determine the relative computational benefits of second-order and first-order methods for the convex PHEV energy management formulation, and the second contribution is a set of numerical studies where the performance of the projected interior point algorithm is compared with the ADMM algorithm of \cite{EastCDC2018}. In these studies we demonstrate that the projected interior point method has superior convergence properties, but requires more time to obtain a solution with modest accuracy, and consequently is only suitable for real-time solutions over shorter horizons (in this case fewer than 500 samples). We also demonstrate that both methods are significantly faster than CVX \cite{CVX}, and in ADMM (using an improved implementation from \cite{EastCDC2018}) we demonstrate the first method capable of solving the energy management problem in real time, over long horizons ($\geq$1000 samples) when nonlinear system dynamics are considered and hard limits on both power and state of charge are enforced over the entire horizon.

The paper is organised as follows: in section \ref{section_2} the energy management problem, MPC framework, and convex reformulation are defined, and section \ref{section_interior_point} details the projected interior point method. The ADMM algorithm of \cite{EastCDC2018} is stated in section \ref{section_ADMM}, numerical experiments are presented in section \ref{section_numerical_experiments}, and the paper is concluded in section \ref{section_conclusion}.


\section{Energy Management Problem \& Model Predictive Control Framework}\label{section_2}


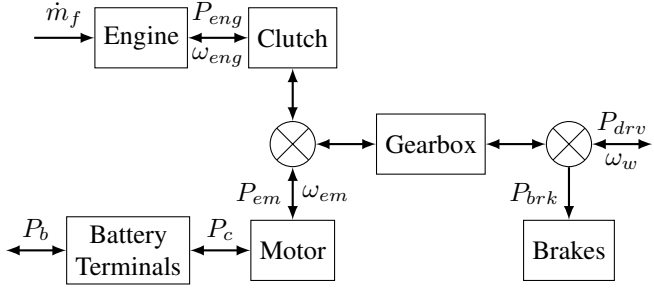
\begin{figure}
\begin{center}
\begin{tikzpicture}[node distance = 0.7cm and 0.8cm]
\node[] (fuel tank) {};
\node[base, right=of fuel tank] (engine) {Engine};
\node[base, right=of engine] (clutch) {Clutch};
\node[add={}{}{}{}, below=of clutch, circle, draw] (torque coupling) {};
\node[base, below=of torque coupling] (motor) {Motor};
\node[base, left=of motor] (battery) {\shortstack{Battery \\ Terminals}};
\node[left=of battery] (empty) {};
\node[base, right=of torque coupling] (gearbox) {Gearbox};
\node[add={}{}{}{}, right=of gearbox, circle, draw] (wheels) {};
\node[base, below=of wheels] (brakes) {Brakes};
\node[right=of wheels] (temp) {};
\draw[arrow, ->] (fuel tank) -- node[above] {$\dot{m}_f$} (engine);
\draw[arrow, <->] (engine) -- node[above] {$P_{eng}$} node[below] {$\omega_{eng}$} (clutch);
\draw[arrow, <->] (clutch) -- (torque coupling);
\draw[arrow, <->] (empty) -- node[above] {$P_b$} (battery);
\draw[arrow, <->] (battery) -- node[above] {$P_c$} (motor);
\draw[arrow, <->] (motor) -- node[left] {$P_{em}$} node[right] {$\omega_{em}$} (torque coupling);
\draw[arrow, <->] (torque coupling) -- (gearbox);
\draw[arrow, <->] (wheels) --  (gearbox);
\draw[arrow, ->] (wheels) -- node[left] {$P_{brk}$} (brakes);
\draw[arrow, <->] (wheels) -- node[above] {$P_{drv}$} node[below] {$\omega_w$} (temp);
\end{tikzpicture}
\caption{A diagram of a simplified model of a parallel PHEV powertrain, labelled with the main flows of power and rotational speeds.}
\label{fig_powertrain_model}
\end{center}
\end{figure}

Fig. \ref{fig_powertrain_model} shows a simplified diagram of a parallel PHEV powertrain, and illustrates the energy transfers that are considered as part of the energy management problem. At a given time, $t$, the mass flow rate of fuel delivered to the engine, $\dot{m}_f$, can be described by a time-varying function, $f$, of engine output power, $P_{eng}$, and engine shaft speed, $\omega_{eng}$, as 
\begin{equation} \label{eqn_fuel_consumption}
\dot{m}_f(t) = f(P_{eng}(t),\omega_{eng}(t),\sigma(t),t),
\end{equation}
where $\sigma$ describes the state of the engine and clutch engagement by
\begin{equation} \label{eqn_engine_switching}
\sigma(t) = \begin{cases} 1 & \text{engine on, clutch engaged} \\
0 & \text{engine off, clutch disengaged} \end{cases}.
\end{equation}
Similarly, the rate of consumption of the battery's internal chemical energy, $P_b$, can be described by a time varying function, $g$, of battery output power (i.e motor input power), $P_c$, which can in turn be described by a time varying function, $h$, of motor output power, $P_{em}$, and motor shaft speed, $\omega_{em}$: 
\begin{equation} \label{eqn_power_dynamics}
\begin{aligned} 
P_b(t) &= g(P_c(t),t)\\ 
P_{c}(t) &= h(P_{em}(t),\omega_{em}(t),t).
\end{aligned}
\end{equation}
Therefore, the state of charge of the battery, $E$, is given at time $t$ by
\begin{align}
E(t) = E(0) - \int_0^t P_b(t) \ \mathrm{d}t. \label{eqn_battery_dynamics}
\end{align}

The engine output power that is delivered through the clutch is combined additively with power from the motor through a coupling device to drive the gearbox. Assuming that all drivetrain components are $100\%$ mechanically efficient, the power delivered to the wheels (i.e the power demanded by the driver), $P_{drv}$, is given by
\begin{equation} \label{eqn_power_coupling}
P_{drv}(t)=P_{em}(t) + P_{eng} (t) + P_{brk}(t),
\end{equation}
where $P_{brk}$ is the power extracted from the system by the mechanical brakes. 
Assuming a discrete variable transmission, the rotational velocities of the engine and motor shafts are given as a function of the rotational velocity of the wheels, $\omega_w$, as
\begin{equation} \label{eqn_gear_shifting}
\omega_{em}(t) = r(t) \omega_{w}(t), \ \omega_{eng}(t) = \sigma(t) r(t) \omega_{w}(t)
\end{equation}
where $r: \mathbb{R} \rightarrow \{r_1,\dots,r_{N_g}\}$, and $N_g$ is the number of available gear ratios. 

The engine has upper and lower limits on torque, $\overline{T}_{eng}$ and $\underline{T}_{eng}$, that are functions of engine speed, so limits on engine power are given by
\begin{equation} \label{eqn_engine_power_limits}
\begin{aligned}
P_{eng}(t) \geq \underline{P}_{eng}(t) &= \underline{T}_{eng}(\omega_{eng}(t))\omega_{eng}(t) \\
P_{eng}(t) \leq \overline{P}_{eng}(t) &= \overline{T}_{eng}(\omega_{eng}(t))\omega_{eng}(t).
\end{aligned}
\end{equation}
The limits on motor power can be given similarly as
\begin{equation} \label{eqn_motor_power_limits}
\begin{aligned}
P_{em}(t) \geq \underline{P}_{em}(t) &= \underline{T}_{em}(\omega_{em}(t))\omega_{em}(t) \\
P_{em}(t) \leq \overline{P}_{em}(t) &= \overline{T}_{em}(\omega_{em}(t))\omega_{em}(t),
\end{aligned}
\end{equation} 
where $\overline{T}_{em}$ and $\underline{T}_{em}$ are upper and lower limits on motor torque. The engine and motor have static limits on rotational speed, given by
\begin{equation}
\underline{\omega}_{em} \leq \omega_{em}(t) \leq \overline{\omega}_{em} \quad \text{and} \quad \underline{\omega}_{eng} \leq \omega_{eng}(t) \leq \overline{\omega}_{eng},
\end{equation}
and the battery has static limits on state-of-charge and rate of charge and discharge, given by
\begin{equation} \label{eqn_battery_limits}
\begin{aligned}
\underline{E} \leq E(t) \leq \overline{E} \quad \text{and} \quad \underline{P}_b \leq P_b(t) \leq \overline{P}_b.
\end{aligned}
\end{equation}

The above system is under-constrained in three degrees of freedom: the fraction of total driver demand power delivered from motor, engine, and brakes; the engine switching and clutch engagement; and the gear selection. Consequently, the parameters $P_{eng}(t)$, $P_{brk}(t)$, $\sigma (t)$, and $r(t)$ must be actively controlled ($P_{em}(t)$ is given in terms of $P_{eng}(t)$ and $P_{brk}(t)$ by (\ref{eqn_power_coupling})). The energy management problem can therefore be written as an open-loop optimal control problem for a journey of length $T$ as 
\begin{equation} \label{eqn_energy_management_problem}
\begin{aligned} 
\min_{\sigma(t),r(t),P_{eng}(t),P_{brk}(t)} & \int_0^T \dot{m}_f(t) \ \mathrm{d}t \\
\text{s.t.} \ &  (\ref{eqn_fuel_consumption})-(\ref{eqn_battery_limits}) \ \forall  t \in [0,T]
\end{aligned}
\end{equation}
Note that $\omega_w$ and $P_{drv}$ (which determine the vehicle's speed and acceleration) are not affected in the control problem; the principle of the controller is to always meet the powertrain output demanded by the driver, and to not affect the overall driving behaviour of the vehicle. 



\subsection{MPC Framework}

If implemented in a real vehicle, the solution found from  (\ref{eqn_energy_management_problem}) will be suboptimal as it is impossible to model the powertrain components with complete accuracy, and because the problem is dependent on future disturbance variables, $P_{drv}$ and $\omega_w$, that are impossible to exactly predict \textit{a priori}. A MPC framework can be used to reduce these limitations, where instead of solving a single instance of the open-loop control problem, the control variables are repeatedly updated as the journey progresses \cite{Buerger}. This allows the predictions of driver behaviour to be improved as new information becomes available, and provides feedback on the vehicle state (i.e the battery state of charge), thus providing a degree of robustness to modelling and prediction errors. We describe a MPC framework for the energy management problem in this section.

Throughout the following text, the notation $\hat{x}$ is used for a variable, $x(t)$, to refer to the discretely sampled prediction used within the MPC framework, as opposed to the physical signals and states described in the previous section. At each control variable update instant, a discretely sampled prediction of demand power and wheel speed is made as
\begin{align*}
\hpdrv &= (\hp_{drv,0},\dots,\hp_{drv,N-1}), \ \hat{\omega}_w = (\hat{\omega}_{w,0},\dots,\hat{\omega}_{w,N-1})
\end{align*}
for $k = 0,\dots,N-1$, where the sampling period $\delta$ is assumed to be constant, and the prediction horizon is given by $N$. Although the prediction aspect of the energy management problem is still very much an open issue \cite{Zhou2019}, the focus of this paper is on the subsequent optimization problem, so it is assumed that an accurate method is available to the controller and this aspect is not addressed further. The engine and motor loss maps can be approximated with quasi-static quadratic functions, $\hat{f}$ and $\hat{h}$, \cite{Hadj-Said2016,Buerger,Nuesch2014} as
\begin{equation} \label{eqn_engine_and_motor_models}
\begin{aligned}
 &\hat{\dot{m}}_{f,k} = \hf (\hpengk,\hat{\omega}_{eng,k},\hat{\sigma}_k) \\
= & \hat{\sigma}_k [ \alpha_2(\hat{\omega}_{eng,k})\hpengk^2 + \alpha_1(\hat{\omega}_{eng,k}) \hpengk + \alpha_0(\hat{\omega}_{eng,k}) ] \\
& \hat{P}_{c,k} = \hh (\hpemk,\hat{\omega}_{em,k}) \\
= & \beta_2(\hat{\omega}_{em,k})\hpemk^2 + \beta_1(\hat{\omega}_{em,k})  \hpemk  + \beta_0(\hat{\omega}_{em,k})
\end{aligned}
\end{equation}
where $\alpha_2(\hat{\omega}_{eng,k}) > 0$ $\forall$ $\hat{\omega}_{eng,k}$, and $\beta_2(\hat{\omega}_{em,k}) > 0$ $\forall$ $\hat{\omega}_{em,k}$. The battery is commonly modelled as an equivalent circuit of internal resistance \cite{Borhan2012a,Xiang2017,Hadj-Said2016,Buerger,EastCDC2018,Nuesch2014,Elbert2014} as:
\begin{align*}
\hpbk = & \hg (\hpemk, \hat{\omega}_{em,k},k) \\
=& \frac{V_{oc,k}}{2 R_k} \left( 1 - \sqrt{1 - \frac{4R_k}{V_{oc,k}^2} \hh (\hpemk,\hat{\omega}_{em,k}) } \right)
\end{align*}
where $V_{oc,k}$ and $R_k$ are the open circuit voltage and internal resistance. The limits on engine, motor, and discharge power (\ref{eqn_engine_power_limits},\ref{eqn_motor_power_limits}) are now given in discrete time as
\begin{equation*} 
\begin{aligned}
\upengk &=  \underline{T}_{eng}(\hat{\omega}_{eng,k}) \hat{\omega}_{eng,k} \\
\overline{P}_{eng,k} &=  \overline{T}_{eng}(\hat{\omega}_{eng,k}) \hat{\omega}_{eng,k} \\
\upemk &= \underline{T}_{em}(\hat{\omega}_{em,k})\hat{\omega}_{em,k} \\
\overline{P}_{em,k} &= \overline{T}_{em}(\hat{\omega}_{em,k})\hat{\omega}_{em,k} \\
\end{aligned}
\end{equation*}
and the MPC optimization at time $t$ is then given as an approximation of (\ref{eqn_energy_management_problem}) with Euler method integration of the state dynamics (\ref{eqn_battery_dynamics}) as 
\begin{equation} \label{eqn_MPC_problem}
\begin{aligned} 
\min_{\hat{\sigma},\hat{r},\hpeng,\hat{P}_{brk}} & \sum_{k=0}^{N-1} \delta{ \hat{\dot{m}}_{f,k} } \\
\text{s.t.} \ & \ \hat{E}_0 = E(t) \\
&  \begin{rcases} \hE_{k+1} = \hEk - \delta \hg(\hpemk,\hat{\omega}_{em,k},k) \\
 \hpdrvk = \hpemk + \hpengk + \hat{P}_{brk,k}\\
 \hat{\omega}_{em,k} = \hat{r}_k \hat{\omega}_{w,k} \\
 \hat{\omega}_{eng,k} = \hat{\sigma}_k\hat{r}_k\hat{\omega}_{w,k} \\
 \hat{r}_k \in \{r_1,\dots,r_{N_g}\} \\
 \hat{\sigma}_k \in \{0,1\} \\
 \upengk \leq \hpengk \leq \overline{P}_{eng,k} \\
 \upemk \leq \hpemk \leq \overline{P}_{em,k} \\
 \underline{\omega}_{em} \leq \hat{\omega}_{em,k} \leq \overline{\omega}_{em}  \\ 
 \underline{\omega}_{eng} \leq \hat{\omega}_{eng,k} \leq \overline{\omega}_{eng}  \\
   \underline{P}_{b} \leq \hat{P}_{b,k} \leq \overline{P}_{b} \\
 \underline{E} \leq \hE_{k+1} \leq \overline{E} \\
\end{rcases} \forall k
\end{aligned}
\end{equation}
At each control variable update instance, the first elements of the vectors of optimization variables, $\hat{\sigma}^\star_0$, $\hat{r}^\star_0$, $\hat{P}_{eng,0}^\star$, and $\hat{P}_{brk,0}^\star$ are implemented as $\sigma(t)$, $r(t)$, $P_{eng}(t)$, and $P_{brk}(t)$, where $\hat{\sigma}^\star$, $\hat{r}^\star$, $\hat{P}_{eng}^\star$, and $\hat{P}_{brk}^\star$ are the minimizing arguments of (\ref{eqn_MPC_problem}). 


\subsection{Convex Reformulation}\label{section_convex_reformulation}

Whilst the MPC framework provides a degree of robustness to prediction and modelling errors, problem (\ref{eqn_MPC_problem}) is challenging to solve as the cost function is non-convex, it is subject to nonlinear, non-convex constraints, and it has $2N$ discrete decision variables ($N$ is likely to be in the thousands for journeys longer than 15 mins, assuming an update frequency of approximately 1Hz). The global minimum of the problem can be found using Dynamic Programming, however this is a computationally demanding approach and not suitable for an online solution \cite{EastCDC2018}. Instead, it is possible to reduce the complexity of the problem by determining the three variables $\hat{\sigma}$, $\hat{r}$, and $\hat{P}_{brk}$ using heuristic rules \cite{Buerger} or an external optimization routine \cite{Hadj-Said2016}, leaving the power balance between the engine and motor as the only under-constrained parameter. It has previously been demonstrated in \cite{EastCDC2018} that the resulting problem is convex when using the battery power, $\hpb$, as the decision variable. This is the approach that is taken in this paper, as discussed below.

Assuming an external method for estimating the engine switching behaviour, the set of timesteps at which the engine is switched on and the clutch is engaged is given by
\begin{equation*}
\p = \{k : \hat{\sigma}(k) = 1 \},
\end{equation*}
and also assuming a method for pre-determining the use of mechanical brake, we re-define the drive demand power from (\ref{eqn_power_coupling}) as
\begin{equation*}
\hp_{drv,k} = \hp_{em,k} + \hp_{eng,k},
\end{equation*}
so that $P_{brk,k}$ is no longer an optimization variable in (\ref{eqn_MPC_problem}). Finally, by assuming that the gear selection, $\hat{r}$, is also pre-determined, $\hat{\omega}_{eng,k}$ and $\hat{\omega}_{em,k}$ are defined for all $k$, so the engine and motor models (\ref{eqn_engine_and_motor_models}) can be reduced to time-varying polynomial functions of the form
\begin{align*}
\hat{\dot{m}}_f(k) & = \hfk(\hpengk) \\
& = \ak2\hpengk^2 + \ak1\hpengk + \ak0, \\
\hat{P}_c(k) & = \hhk (\hpemk) \\
& = \bk2 \hpemk^2 + \bk1\hpemk + \bk0.
\end{align*}
These functions are strictly convex ($\ak2,\bk2 > 0$), and we then ensure that $\hgk$ is also strictly convex by assuming that $V_{oc}$ and $R$ are independent of state of charge (and in this case, constant), so that the battery model is approximated by
\begin{align*}
& \hpbk  = \hgk(\hpemk) = \frac{V^2_{oc}}{2 R} \left( 1 - \sqrt{1 - \frac{4R}{V_{oc}^2} \hhk(\hpemk ) } \right).
\end{align*}
We update the limits on $\hpengk$ and $\hpemk$ to ensure that  $\hat{f}_k$, $\hat{h}_k$, and $\hat{g}_k$ are all non-decreasing and real-valued as
\begin{align*}
\upengk & = \max  \left\{ \upengk, -\frac{\ak1}{2 \ak2} \right\}\\
\upemk & = \max \left\{ \upemk, -\frac{\bk1}{2 \bk2} \right\} \\
\opemk & = \min\bigg\{  \opemk, \max\left\{ x : 1 - \frac{4R}{V_{oc}^2} \hhk(x ) = 0 \right\} \bigg\}.
\end{align*}
This also ensures that $\hgk$ is a one-to-one function, so we can define
\begin{equation*}
\hpbk = \hgk (\hpemk) \Leftrightarrow \hpemk = \hgk^{-1} (\hpbk)
\end{equation*}
where
\begin{align*}
&\hgkinv(\hpbk ) = -\frac{\bk1}{2\bk2} + \sqrt{-\frac{R\hpbk^2}{\bk2 V_{oc}^2} + \frac{\hpbk - \bk0}{\bk2} + \frac{\bk1^2}{4\bk2^2} }
\end{align*}
Using this definition of $\hgkinv$, it is known that for $k \in \p$,
\begin{equation*}
\hpengk = \hpdrvk - \hgkinv (\hpbk)
\end{equation*}
where the corresponding limits on $\hpb$ are
\begin{align*}
\opbk &= \min \{  \opb, \hgk(\opemk), \hgk (\hpdrvk - \upengk ) \} \\
\upbk &= \max \{\upb ,\hgk(\upemk), \hgk (\hpdrvk - \opengk ) \}
\end{align*}
and it is known that for $k \notin \p$
\begin{align*}
\hfk(\hpengk) = &0, \quad \opbk = \upbk = \hgk(\hpdrv(k)).
\end{align*}

In \cite{EastCDC2018}, it is shown that given the properties of $\hgk(\cdot)$ (strictly convex, twice differentiable, non-decreasing, one-to-one) and $\hfk$ (convex and non-decreasing), the function $\hfk(\hpdrvk -\hgkinv (\hpbk))$ is convex, non-increasing, and twice differentiable, so the MPC problem (\ref{eqn_MPC_problem}) becomes the convex, linearly constrained optimization problem
\begin{equation} \label{eqn_convex_problem}
\begin{aligned} 
\min_{\hpb} & \sum_{k \in \p} \delta\hfk(\hpdrvk - \hgkinv(\hpbk)) \\
\text{s.t.} \ & \hE_0 = E(t) \\
& \begin{rcases} \hE_{k+1} = \hE_k - \delta \hpbk \\
\underline{E} \leq \hE_{k+1} \leq \overline{E} \end{rcases} k = 0,\dots,N-1 \\
& \upbk \leq \hpbk \leq \opbk  && k \in \p \\ 
& \upbk = \hpbk = \opbk  && k \notin \p. \\ 
 \end{aligned}
\end{equation}

For the sake of clarity in the following sections, we now revert to the commonly used notation for MPC problems, where $u$ is the control input (the predicted battery power, $\hpb$), and $x$ is the state variable (the predicted state of charge, $\hE$), i.e
\begin{align*}
u  &:=\hpb \in \mathbb{R}^N, & \ou  &:= \upb \in \mathbb{R}^N, & \uu &:= \upb \in \mathbb{R}^N \\
x &:= \hE \in \mathbb{R}^N, &  \ox &:= \overline{E} \in \mathbb{R}, & \ux &:= \underline{E} \in \mathbb{R}.
\end{align*}
Then, by defining the non-increasing, separable, strictly convex cost function
\begin{align*}
F(u) = \sum_{k \in \p} \delta\hfk(\hpdrvk - \hgkinv(u_k)),
\end{align*}
we can express (\ref{eqn_convex_problem}) equivalently as
\begin{equation} \label{eqn_convex_problem_new}
\begin{aligned} 
\min_{u} \ & F(u) \\
\text{s.t.} \ & x = \Phi x_0 - \Psi u \\
&\Phi \ux \leq x \leq \Phi \ox && \\
&  \uu \leq u \leq \ou, \\ 
 \end{aligned}
\end{equation}
where $\Phi$ is a vector of $N$ ones, and $\Psi$ is an $N \times N$ lower triangular matrix where every non-zero element is equal to $\delta$. 

If we describe $\F_k$ as the feasible set of state of charge values at timestep $k$, then problem (\ref{eqn_convex_problem}) is feasible if and only if $\F_k \neq \emptyset$ for all $k = 0,\dots,N$, where $\F_{k}$ is defined recursively by
\begin{align*}
\F_{k+1} &= \{x_k - \delta u_k \in \X :x_k \in \F_k, \ u_k \in \U_k \}\\
&=  \{ \F_k \oplus -\delta \U_k \} \cap \X,
\end{align*}
and
\begin{align*}
\U_k &= \{u_k: \uu_k \leq u_k \leq \ou_k \} \\
\X &= \{x_k: \underline{x} \leq x_k \leq \overline{x} \}.
\end{align*}
As $\F_k$ is a one-dimensional convex set, it can be parameterized by $\F_k = [\min \F_k, \max \F_k ]$, and the sequence $\F_{1}, \dots, \F_{N}$ can be obtained iteratively from 
\begin{equation}\label{feasability_check}
\begin{aligned}
\max \F_{k+1} &= \min\{\overline{x}, \max \F_k - \delta \uu_k \} \\
\min \F_{k+1} &= \max\{\underline{x}, \min \F_k - \delta \ou_k \}
\end{aligned}
\end{equation} 
with $\F_0=\{x_0\}$. The problem is therefore feasible if and only if $\max \F_{k+1}$ and $\min \F_{k+1}$ exist (i.e. $\max \F_{k+1} \geq \min \F_{k+1}$ when calculated from (\ref{feasability_check}), and $\overline{u}_k \geq \underline{u}_k$) for $k \in \{0,\dots,N-1\}$. 

The cost function in (\ref{eqn_convex_problem_new}) is known to be non-increasing in $u$, so by inspection, if 
\begin{equation*}
\Phi \underline{x} \leq \Phi x_0 - \Psi \overline{u} \leq \Phi \overline{x},
\end{equation*}
then $\overline{u}$ is the minimizing argument.

\section{Projected Interior Point Method}\label{section_interior_point}
Problem (\ref{eqn_convex_problem_new}) is in the form of a convex nonlinear program with affine equality and inequality constraints. Here we present a projected primal-dual interior point algorithm, where the element-wise bounds on the control variable are applied as a projection. This section begins with the definition of the barrier approximation and optimality conditions, followed by a statement of the initialization algorithm and main projected interior point algorithm with accompanying pseudocode. The section is then is then concluded with an analysis of the computational complexity of each iteration, and a discussion of convergence to the minimizing argument of (\ref{eqn_convex_problem_new}).

Firstly, we introduce a slack variable $s \in \mathbb{R}^{2N}$, and write problem (\ref{eqn_convex_problem_new}) equivalently as\begin{equation} \label{eqn_ip_problem_1}
\begin{aligned} 
\min_{u} \ & F(u) \\
\text{s.t.} \ & A u - b - s = 0 \\ 
& s \geq 0 \\
& \uu \leq u \leq \ou, \\ 
 \end{aligned} 
\end{equation}
where
\begin{align*}
A = \begin{bmatrix} \Psi \\ - \Psi \end{bmatrix} ,  \ \text{and} \ \ b = \begin{bmatrix} \Phi(x_0 - \ox) \\ - \Phi(x_0 - \ux)  \end{bmatrix}.
\end{align*}
This can then be approximated with 
\begin{equation} \label{eqn_ip_problem_2}
\begin{aligned} 
\min_{u,s} \ & F(u) + B(s,\mu ) \\
\text{s.t.} \ & A u - b - s = 0 \\ 
& \uu \leq u \leq \ou, \\ 
 \end{aligned} 
\end{equation}
where $B(s,  \mu )$ is a log barrier function defined by
\begin{align*}
B(s, \mu ) = -\frac{1}{\mu} \sum_{k=1}^{2N} \log s_k,
\end{align*}
and $\mu>0$ can be interpreted as the degree to which the log barrier function approximates the inequality constraint, $s\geq 0$. The Lagrangian function associated with problem (\ref{eqn_ip_problem_2}) is
\begin{multline*}
\mathcal{L}(u,s,\theta_1,\theta_2,\theta_3, \mu) = F(u) + B(s, \mu ) - \theta_1^\top(Au - b - s)\\  - \theta_2^\top (u - \uu) - \theta_3^\top (\ou - u)
\end{multline*}
where $\theta_1\in\mathbb{R}^{2N}$, $\theta_2 \in \mathbb{R}^N$, and $\theta_3 \in \mathbb{R}^N$. By defining the set
\begin{align*}
\A^\con = \{ & k : u^\con_k = \uu_k,  \nabla_k  F(u^\con) - A_k^\top \theta_1^\con > 0, \\
& \text{or}  \ u^\con_k = \ou_k, \ \nabla_k  F(u^\con) - A_k^\top \theta_1^\con < 0 \},
\end{align*}
the necessary and sufficient conditions for the optimal values $u^\con,s^\con$, and $\theta_1^\con$ that minimize (\ref{eqn_ip_problem_2}) are
\begin{subequations}\label{KKT2}
\begin{align}
\nabla_k F(u^\con) - A_k^\top \theta_1^\con & = 0, \quad k \notin \A^\con, \label{KKT2_a}\\
S^\con \theta_1^\con & = \frac{1}{\mu}\textbf{1}, \label{KKT2_b} \\
Au^\con - b - s^\con &= 0, \label{KKT2_c}\\
u^\con - \uu &\geq 0,  \label{KKT2_d}\\
\ou - u^\con & \geq 0,  \label{KKT2_f}\\
s^\con &> 0, \label{KKT2_g}\\
\theta_1^\con &\geq 0, \label{KKT2_h}
\end{align}
\end{subequations}
where $\nabla_kF(u)$ is the $k$th row of the gradient of $F$ at $u$, $A_k$ is the $k$th column of $A$, and $S = \text{diag}(s)$.

It can be demonstrated that to obtain the optimality conditions of  (\ref{eqn_ip_problem_1}), two changes must be made to (\ref{KKT2}): firstly, (\ref{KKT2_g}) must become a non-strict inequality condition, and secondly, a vector of zeros must replace $\frac{1}{\mu}\textbf{1}$ in the R.H.S of (\ref{KKT2_b}). Therefore, it can be concluded that $u^\con$ converges asymptotically to $u^\star$ as $\mu \to \infty$, where $u^\star$ is the minimizing argument of (\ref{eqn_ip_problem_1}).
The principle of the algorithm presented in this section is that conditions (\ref{KKT2_d}-\ref{KKT2_h}) hold at all iterations, whilst a projected Newton method \cite{Optimization1982} is used to obtain an approximation of the solutions of (\ref{KKT2_a}-\ref{KKT2_c}) for a fixed value of $\mu$. This is then repeated for progressively larger values of $\mu$, with progressively higher accuracy, to obtain $u^\star$. 

\subsection{Initialization}

The projected interior point algorithm is initialized using Algorithm \ref{algorithm_intialisation} to obtain the tube defined by the sequence $ \G_0, \dots, \G_N$ where 
\begin{align*}
\G_k &= \{x_{k+1} + \delta u_k \in \F_k : x_{k+1} \in \G_{k+1}, u_k \in \U_k\} \\
&= \F_k \cap \left\{ \G_{k+1} \oplus \delta \U_k \right\} .
\end{align*}
The centerline of this tube is then used to obtain values for $u^{(0)}$, $s^{(0)}$, and $\theta_1^{(0)}$ that satisfy (\ref{KKT2_b}-\ref{KKT2_h}).

\begin{algorithm}
\caption{Initialization Algorithm\label{algorithm_intialisation}}
\begin{algorithmic}[1]
\STATE Set $\max \F_0= \min \F_0=x_0$
\FOR {$k = 0,\dots,N-1$}
\STATE $\max \F_{k+1} = \min\{\overline{x}, \max \F_k - \delta \uu_k \}$
\STATE $\min \F_{k+1} = \max\{\underline{x}, \min \F_k - \delta \ou_k \}$
\ENDFOR
\STATE $\G_N$ = $\F_N$
\STATE $x_N^{(0)} = \frac{1}{2}(\max \G_N + \min \G_N )$
\FOR {$k = N-1,\dots,0$}
\STATE $\max \G_k = \min\{\max \G_{k+1} + \delta \ou_k, \max \F_k \}$
\STATE $\min \G_k = \max\{\min \G_{k+1} + \delta \uu_k, \min \F_{k} \}$
\STATE $x^{(0)}_k = \frac{1}{2}(\max \G_k + \min \G_k)$
\STATE $u^{(0)}_k = \frac{1}{\delta} ( x^{(0)}_k - x^{(0)}_{k+1}).$
\ENDFOR
\STATE $s^{(0)} = A u^{(0)} - b$
\STATE $\theta_1^{(0)} = \frac{1}{\mu}(S^{(0)})^{-1}\textbf{1}$
\end{algorithmic}
\end{algorithm}

\begin{proposition}
The values of $u^{(0)}$, $s^{(0)}$, and $\theta_1^{(0)}$ obtained with Algorithm \ref{algorithm_intialisation} will satisfy conditions (\ref{KKT2_b}-\ref{KKT2_h}), iff
\begin{equation}\label{new_init_cond}
\begin{aligned}
\max \F_k &> \min \X \quad \text{and} \\
\min \F_k &< \max \X
\end{aligned}
\end{equation}
for $k \in {1,\dots,N}$.
\end{proposition}
\begin{proof}
By induction. We start by defining the operation
$$
|\mathcal{S}| = \max \mathcal{S} - \min \mathcal{S}
$$
for any given set $\mathcal{S}$, and note that for $\max \F_k$ and $\min \F_k$ to exist in (\ref{new_init_cond}) it is implied that the problem is feasible, so $|\F_k| \geq 0$ for $k \in \{1, \dots ,N \}$. Let $k^\dagger$ be the smallest value of $k$ such that $\max \F_{k} = \max \X$ and/or $\min \F_{k} = \min \X$, and suppose that $k^\dagger \leq N$. In this case we can use (\ref{new_init_cond}) to show that $\max \F_k > \min \F_k$ (i.e. $ |\F_k|>0$). Furthermore, we know that $|\F_{k+1}| \geq |\F_k|$ for any $k$ where $\max \F_{k+1} \neq \max \X$, $\min \F_{k+1} \neq \min \X$, and $|\U_k| \geq 0$. Therefore $|\F_{k}|>0$ for all $k\geq k^\dagger$. If we assume that $|\F_k|>0$, and $|\G_{k+1}|>0$, we can then show that
\begin{align}
\left| \F_{k+1} \right| > 0 \quad \Leftrightarrow \quad & \left|  \{ \F_k \oplus - \delta \U_k \} \cap \X \right| > 0 \nonumber \\
\Rightarrow \quad & \left|  \F_k \cap \{ \F_{k+1} \oplus \delta \U_k \}  \right| > 0 \nonumber \\
\Rightarrow \quad & \left|  \F_k \cap \{ \G_{k+1} \oplus \delta \U_k \} \right| > 0  \nonumber \\
\Rightarrow \quad & \left| \G_k \right| >0. \label{feas_proof}
\end{align} 
In the case where $k^\dagger$ exists we know that $|\F_N|>0$, so $|\G_N|>0$, and the argument in (\ref{feas_proof}) can then be made recursively to show that $|\G_k|>0$ for $k = N-1,\dots,k^\dagger$. As $\G_k \subseteq \F_k \subseteq \X$, we now know that that $x_k^{(0)} \in \text{int} (\X)$ for $k \in \{k^\dagger,\dots,N\}$. 

For $k \in \{1,\dots,k^\dagger-1\}$ we know that $|\F_k| \geq 0$ from feasibility, and using a similar method to (\ref{feas_proof}) we can then show that $|\F_{k+1}| \geq 0 \Rightarrow |\G_k| \geq 0$. We also know that $\max \F_k \neq \max \X$ and $\min \F_k \neq \min \X$ for $k \in \{1,\dots,k^\dagger-1\}$ (from the definition of $k^\dagger$), so as $\G_k \subseteq \F_k \subset \X$, it follows that $x_k^{(0)} \in \text{int} (\X)$ for $k \in \{1,\dots,k^\dagger-1\}$. A similar result can be shown for $k \in \{1,\dots,N\}$ if $k^\dagger$ does not exist. 

The condition $x_k^{(0)} \in \text{int} (\X)$ for $k \in \{1,\dots,N\}$ ensures that $Au^{(0)} - b > 0$, so $s^{(0)} = A u^{(0)} - b$ ensures (\ref{KKT2_c}) and (\ref{KKT2_g}), and $\theta_1^{(0)} = \frac{1}{\mu}(S^{(0)})^{-1}\textbf{1}$ ensures (\ref{KKT2_b}) and (\ref{KKT2_h}). Finally, it can also be shown that 
\begin{align*}
&\max u_k^{(0)} \\
 =& \max \frac{1}{2 \delta} \left\{ \max \G_k + \min \G_k - \max \G_{k+1} - \min \G_{k+1} \right\} \\
 =& \overline{u}_k
\end{align*}
and we can similarly show that $\min u_k^{(0)} = \underline{u}_k$, which therefore demonstrates (\ref{KKT2_d}-\ref{KKT2_f}) if (\ref{new_init_cond}) is true for $k \in \{1,\dots,N\}$. 

Conversely, a subset of the above results are shown to be not true if (\ref{new_init_cond}) is not true for some $k \in \{1,\dots,N\}$.
\end{proof}


An example of the $x^{(0)}$ and $u^{(0)}$ values obtained by the initialization algorithm for a nominal, randomly generated example is shown in Fig. \ref{fig_tube}, demonstrating that $\underline{u}_k \leq u_k^{(0)} \leq \overline{u}_k \ \forall k$, and that $\underline{x} < x^{(0)}_k < \overline{x} \ \forall k$.

\begin{figure}
\includegraphics[scale=1]{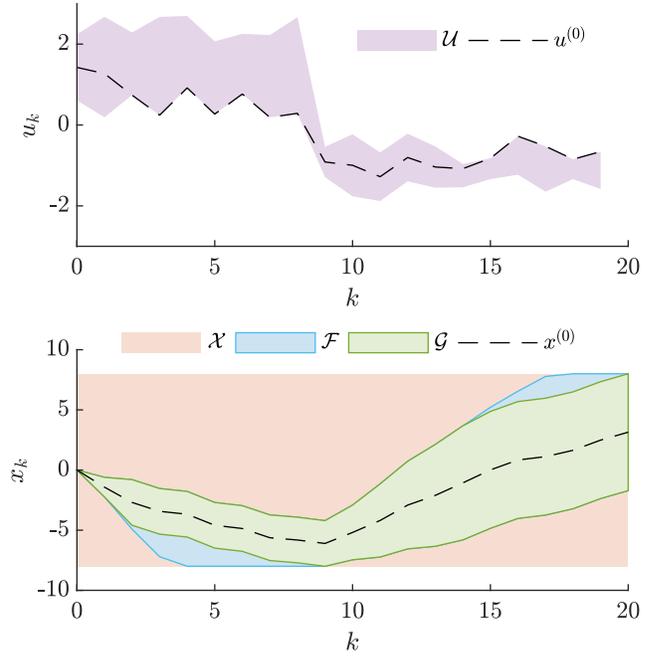}
\caption{An example of the tubes defined by $\F$ and $\G$, and the solutions obtained for $x^{(0)}$ and $u^{(0)}$ using Algorithm \ref{algorithm_intialisation} for a nominal system (the values were chosen to illustrate the operation of the algorithm, and are not neccessarily representative of those observed in the energy management problem). \label{fig_tube}}
\end{figure}

In addition to $u$, $s$, and $\theta_1$, there are three further parameters that are initialized at the start of the algorithm: $\mu_0>0$, $\overline{\mu} \geq \mu_0$, $k_\mu>1$, and $\tau \in (0,1)$. These parameters can be assigned any value within the stated ranges, and their significance is discussed in the following subsection.

\subsection{Algorithm}\label{section_algorithm}

At each iteration, $j$, the elements $k$ are partitioned into the sets
\begin{alignat*}{2}
\A^{(j)} & = \{ && k : u^{(j)}_k = \uu_k, \nabla_k F(u^{(j)}) - A_k^\top \theta_1^{(j)} > 0, \\
& && \text{or}  \ u^{(j)}_k = \ou_k, \ \nabla_k F(u^{(j)}) - A_k^\top \theta_1^{(j)} < 0 \} \\
\mathcal{D}^{(j)} & = \{ && k : k \notin \mathcal{A}^{(j)} \},
\end{alignat*}
then an estimate is made of the solution to the equations (\ref{KKT2_a}-\ref{KKT2_c}) using a projected Newton method. Let
\begin{multline}\label{linear_system_of_equations}
\begin{bmatrix} \nabla^2_{ \mathcal{D} } F(u^{(j)}) & 0 & -w^\top  \\
0 & \Theta_1^{(j)} & S^{(j)} \\
 w & -I & 0  \end{bmatrix} \begin{bmatrix} \Delta \tilde{u}^{(j)} \\ \Delta {s}^{(j)} \\ \Delta \theta_1^{(j)}  \end{bmatrix} \\
 = \begin{bmatrix} - \nabla_{ \mathcal{D} } F(u^{(j)}) + v^\top \theta_1^{(j)}  \\
\frac{1}{\mu}\textbf{1} - S^{(j)} \theta_1^{(j)} \\
-A u^{(j)} + b + s^{(j)}  \end{bmatrix}
\end{multline}
where $\nabla^2_{ \mathcal{D} } F(u^{(j)})$ is the $k \in { \mathcal{D}^{(j)} }$ rows and columns of the Hessian of $F$ evaluated at $u^{(j)}$, $\nabla_{ \mathcal{D} } F(u^{(j)})$ is the $k \in { \mathcal{D}^{(j)} }$ elements of the gradient of $F$ evaluated at $u^{(j)}$, and $\Theta = \text{diag}(\theta)$. We define $\hat{A}$ as the matrix $A$ with the $k \in \{k: \uu_k = u^{(j)}_k \ \text{or} \ \ou_k = u^{(j)}_k \}$ columns set to zero, so that $v$ is the $k\in \mathcal{D}^{(j)}$ columns of $A$, and $w$ is the $k\in \mathcal{D}^{(j)}$ columns of $\hat{A}$. The search directions $\Delta \tilde{u}^{(j)}$, $\Delta \tilde{s}^{(j)}$, and $\Delta \theta_1^{(j)}$ are obtained from the reduced equations
\begin{subequations}\label{linear_system_solutions}
\begin{align}
\Delta \theta_1^{(j)} =& \left( w  \left(\nabla^2_{ \mathcal{D}} F (u^{(j)}) \right)^{-1}   w^\top + (\Theta_1^{(j)})^{-1}S^{(j)}\right)^{-1}  \nonumber \\
& \bigg(\frac{1}{\mu}(\Theta_1^{(j)})^{-1} \textbf{1} -  Au^{(j)} + b \nonumber \\
& -  w  \left(\nabla^2_{ \mathcal{D}} F (u^{(j)}) \right)^{-1} (-\nabla F_{  \mathcal{D}}(u^{(j)}) + v^\top \theta_1^{(j)}) \bigg) \label{linear_system_solution1}\\
\Delta \tilde{u}^{(j)} = & -\big( \nabla^2 F_{ \mathcal{D}}(u^{(j)}) \big)^{-1} \nonumber \\
&(\nabla_{ \mathcal{D}} F(u^{(j)}) - v^\top \theta_1^{(j)} - w^\top \Delta \theta_1 ), \label{linear_system_solution2} \\
\Delta s^{(j)} = & A u^{(j)} + w \Delta \tilde{u}^{(j)} - b - s^{(j)}, \label{linear_system_solution3}
\end{align}
\end{subequations}
and search lengths $\alpha_s$ and $\alpha_\theta$ are the determined from the `fraction to the boundary' rule, {\cite[pp.~567]{NumericalOptimization}} as
\begin{subequations} \label{fraction_to_the_boundary}
\begin{align}
\alpha_s & = \max \{ \alpha \in (0,1]:s^{(j)} + \alpha \Delta s_1^{(j)} \geq (1 - \tau) s^{(j)} \}, \label{alpha_s} \\ 
\alpha_\theta & = \max \{ \alpha \in (0,1]:\theta_1^{(j)} + \alpha \Delta \theta_1^{(j)} \geq (1 - \tau) \theta_1^{(j)} \}, \label{alpha_theta}
\end{align}
\end{subequations}
where $\tau \in (0,1)$ is fixed and arbitrary. The variables $u$, $s$, and $\theta_1$ are then updated as
\begin{subequations} \label{variable_updates}
\begin{align} 
u^{(j+1)}_k &= \begin{cases} \pi^u_k\left[u^{(j)}_k + \alpha_s \Delta \tilde{u}^{(j)}_{i(k)}\right] & k \in \mathcal{D}^{(j)} \\
0 &   k \in \A^{(j)} \end{cases}, \ \label{u_update}\\
s^{(j+1)} &= s^{(j)} + \alpha_s \Delta \tilde{s}, \label{s_update} \\
\theta_1^{(j+1)} &= \theta_1^{(j)} + \alpha_\theta \Delta \theta_1, \label{theta_update}
\end{align}
\end{subequations}
where $\pi_k^u(u_k) = \min \{ \overline{u}_k, \max \{ \underline{u}_k,u \} \}$, and $i(k)$ is a function that returns the index of the element of $\mathcal{D}^{(j)}$ that is equal to $k$, assuming that the elements are ordered in a chronological, increasing sequence (e.g $i(7)=2$ if $ \mathcal{D}^{(j)}  = \{ 1,7,12,\dots\}$). This iteration is performed repeatedly until the criterion
\begin{multline}\label{termination_criteria}
r_{\text{IP}}^{(j)} = \max \Big\{ \|\nabla_{ \mathcal{D}} F(u^{(j)}) - v^\top \theta_1^{(j)} \|, \\\|\frac{1}{\mu}\textbf{1} - S^{(j)}\theta_1^{(j)} \|, \| A u^{(j)} - b - s^{(j)} \| \Big\} < \frac{1}{\mu} 
\end{multline}
is met, at which point the value of $\mu$ is updated using 
\begin{align} \label{mu_update}
\mu = \min \{ \overline{\mu}, k_\mu \mu \},
\end{align}
where $\overline{\mu} > 0$ is a pre-determined upper limit on the value of $\mu$, and $k_\mu > 1$ is a pre-determined, arbitrary constant. The algorithm as a whole is then terminated when both conditions (\ref{termination_criteria}) and $\mu = \overline{\mu}$ are met. Algorithm \ref{algorithm_PPDIPM} presents a pseudocode implementation of the above description.

\begin{algorithm}
\caption{Projected Interior Point Method\label{algorithm_PPDIPM}}
\begin{algorithmic}[1]
\STATE Set parameters $\mu_0>0$, $\overline{\mu} \geq \mu_0$, $k_\mu > 1$, and $\tau \in (0,1)$
\STATE Initialize $u^{(0)}$, $s^{(0)}$, and $\theta_1^{(0)}$ using Algorithm \ref{algorithm_intialisation}
\STATE $\mu \gets \mu_0$ and $j \gets 0$
\STATE Determine $\A^{(0)}$ and $\mathcal{D}^{(0)}$
\REPEAT 
\STATE Calculate $\Delta \theta_1^{(j)}$, $\Delta \tilde{u}^{(j)}$, and  $\Delta s$ using (\ref{linear_system_solution1}-\ref{linear_system_solution2})
\STATE Calculate $\alpha_s$ and $\alpha_\theta$ from (\ref{alpha_s}-\ref{alpha_theta})
\STATE Update $u^{(j+1)}$, $s^{(j+1)}$, and $\theta_1^{(j+1)}$ with (\ref{u_update}-\ref{theta_update})
\STATE $j \gets j + 1$
\STATE Determine $\A^{(j)}$ and $\mathcal{D}^{(j)}$
\IF{{$r_{\text{IP}}^{(j)}  < \frac{1}{\mu}$}}
\STATE Update $\mu = \min \{ \overline{\mu}, k_\mu \mu \}$
\ENDIF
\UNTIL{$\mu = \overline{\mu}$ and $r_{\text{IP}}^{(j)}  < \frac{1}{\mu}$} 
\STATE $u^\star \gets u^{(j)}$
\end{algorithmic}
\end{algorithm}
\subsection{Complexity}\label{IP_complexity}
Lines 6-13 in Algorithm \ref{algorithm_PPDIPM} constitute the recursive elements of the projected interior point algorithm, and Table \ref{projected_complexity} presents an analysis of the complexity of equations (\ref{linear_system_solutions}-\ref{termination_criteria}). The significance of enforcing the bounds on $u$ as a projection is illustrated, as the complexity of the projected interior point operations is a function of the first dimension of $A$, which is $2N$ here. If the bounds on $u$ were applied as log barrier functions (i.e $A = (\Psi,-\Psi,I,-I)$ and $b = (\Phi(x_0 - \overline{x}),-\Phi (x_0 - \underline{x}),\underline{u},-\overline{u})$ in (\ref{eqn_ip_problem_1})) the relevant dimension would instead be $4N$, and the complexity of each update would become $\mathcal{O}(4^n N^n)$.   It can be seen that (\ref{linear_system_solution1}) is the most computationally demanding update due to the presence of both a dense matrix-matrix multiplication and a dense matrix inverse (note that \textit{diagonal} matrix operations, e.g. $(\nabla^2_{ \mathcal{D} } F(u^{(j)}))^{-1}$, are omitted from Table \ref{projected_complexity}). The computational complexity of each iteration of the projected interior point algorithm is therefore $\mathcal{O}(2^nN^n)$, where $n \leq 3$ is determined by the method used for matrix multiplication and inversion.

\begin{table}
\begin{center}
\caption{ Summary of the matrix/vector operations present in each equation used in the projected interior point algorithm, where v refers to a vector, and M refers to a non-diagonal matrix.\label{projected_complexity}}
\begin{tabular}{c c c c c}
Equation(s) & M$^{-1}$ & M$\cdot$M & M$\cdot$v & $\mathcal{O}(2^nN^n)$ \\ \hline
(\ref{linear_system_solution1}) & Yes & Yes & Yes & $n \leq 3$ \\
(\ref{linear_system_solution2}-\ref{linear_system_solution3}) & No & No & Yes & $n \leq 2$ \\
(\ref{fraction_to_the_boundary}-\ref{variable_updates}) & No & No & No & $n=1$ \\
(\ref{termination_criteria}) & No & No & Yes & $n \leq 2$ \\
\end{tabular}
\end{center}
\end{table}
\subsection{Convergence}

The algorithm presented in Section \ref{section_algorithm} can be interpreted as a projected Newton method \cite{Optimization1982} used to obtain a stationary point, ($u^\circ$, $s^\circ$, $\theta_1^\circ$), of the function
\begin{equation}\label{convergence_lagrangian}
\hat{\mathcal{L}}(u,s,\theta_1, \mu ) = F(u) + B(s, \mu ) - \theta_1^\top (Au - b - s)
\end{equation}
for a given value of $\mu$, subject to the constraint $\underline{u} \leq u \leq \overline{u}$. The strict inequality in the definition of $\mathcal{A}^\circ$ means that there is a region of $(u,s,\theta_1)$-space close to $(u^\circ,s^\circ,\theta_1^\circ)$ where $\mathcal{A}^{(j)}=\mathcal{A}^\circ$, and a subset of this region will meet the conditions for local quadratic convergence of Newton's method for nonlinear equations {\cite[pp.~276]{NumericalOptimization}}. This means that the R.H.S of (\ref{linear_system_of_equations}) converges to 0 as $j\to\infty$, and the termination criterion (\ref{termination_criteria}) will be met in a finite number of steps. Global convergence could be ensured by adapting the $\Delta u$ step at each iteration with a line-search {\cite[pp.~30]{NumericalOptimization}} of an appropriate merit function, although the merit function from \cite{Optimization1982} cannot be used as the stationary point ($u^\circ$, $s^\circ$, $\theta_1^\circ$) is not a minimum of the function (\ref{convergence_lagrangian}) in general. Despite this limitation, convergence was demonstrated for all problem classes in the simulations that follow.

We have previously demonstrated that $u^{\con}\to u^\star$ as $\mu$ is increased towards $\infty$, and the value of $u^{(j)}$ when the criterion (\ref{termination_criteria}) is met converges to $u^\con$ as $\mu$ is increased towards $\infty$. Therefore, the value of $u^{(j)}$ when Algorithm \ref{algorithm_PPDIPM} terminates can be made arbitrarily close to the minimizing argument of (\ref{eqn_ip_problem_1}) by setting $\overline{\mu}$ arbitrarily high. The algorithm could be further optimised to update the value of $\mu$, possibly at every iteration, to ensure that the iterate remains in, or at least near, to the superlinearly convergent region around $u^{\con}$, $s^\con$, and $\theta_1^\con$, although in the numerical experiments that follow we demonstrate superlinear convergence for a broad class of problems using the simple update (\ref{mu_update}).


\section{Alternating Direction Method of Multipliers}\label{section_ADMM}
We compare the performance of the projected interior point method in simulation with the ADMM algorithm proposed in \cite{EastCDC2018}, which is restated here with a new complexity analysis. We introduce a dummy variable, $\zeta$, and rewrite (\ref{eqn_convex_problem_new}) as
\begin{equation} \label{eqn_ADMM_problem}
\begin{aligned} 
\min_{u} \ & F(u) + \Lambda(u, x) \\
\text{s.t.} \ & \zeta = -u \\
& x = \Phi x_0 + \Psi \zeta
 \end{aligned}
\end{equation}
where the indicator function, $\Lambda$, is defined by
\begin{align*}
\Lambda(u, x) &= \sum_{k=0}^{N-1}\Lambda_k^{u}(u_k) + \sum_{k=1}^N \Lambda^{x}_k(x_k) \\
\Lambda^{z}_k(z) &= \begin{cases} 0 & \underline{z}_k \leq z_k \leq \overline{z}_k \\ \infty & \text{otherwise} \end{cases},
\end{align*}
and the augmented Lagrangian associated with (\ref{eqn_ADMM_problem}) is
\begin{align*}
L(u, \zeta, x, \lambda_1 ,\lambda_2) =& F(u) + \Lambda(u, x)  + \frac{\rho_1}{2} \| u + \zeta + \lambda_1 \|^2 \\
& + \frac{\rho_2}{2} \| \Phi x_0 + \Psi\zeta - x + \lambda_2 \|^2 
\end{align*}
The ADMM algorithm is initialized with the values
\begin{equation}\label{ADMM_intial_values}
\begin{aligned}
&u^{(0)} = \ou, \ \zeta^{(0)} = -u^{(0)}, \ x^{(0)} = \Pi^x\left( \Phi x_0 + \Psi \zeta^{(0)} \right)\ \\
&\ \lambda_1^{(0)} = 0, \ \lambda_2^{(0)} = \Phi x_0 + \Psi \zeta^{(0)} - x^{(0)}
\end{aligned}
\end{equation}
and by defining projection functions
\begin{align*}
\pi^{z}_k(z) &= \min \{\overline{z}_k, \max \{ \underline{z}_k, z \} \}, \\
\Pi^z(z) &= [\pi^z_1(z_1),\dots,\pi^z_N(z_N)],
\end{align*}
the iteration is given by
\begin{equation}\label{ADMM_iteration}
\begin{aligned}
u^{(j+1)}_k =&  \pi_k^u  \bigg[ \arg \min_{u_k} \hfk(\hpdrvk - \hgkinv(u_k)) \\
& + \frac{\rho_1}{2} (u_k + \zeta^{(j)}_k + \lambda_{1,k}^{(j)})^2   \bigg] && k \in \p \\
u^{(j+1)}_k =& \ou_k && k \notin \p \\
x^{(j+1)}=&  \Pi^x \big[ \Phi x_0 + \Psi \zeta^{(j)} + \lambda_2^{(j)} \big]  \\
 \zeta^{(j+1)} =&  (\rho_1 I + \rho_2 \Psi^\top \Psi)^{-1} \big[ -\rho_1 (u^{(j+1)} + \lambda_1^{(j)}) \\
& -\rho_2 \Psi^\top (\Phi x_0 - x^{(j+1)} + \lambda_2^{(j)}) \big]\\
\lambda_1^{(j+1)} =& \lambda_1^{(j)} + u^{(j+1)} + \zeta^{(j+1)} \\
\lambda_2^{(j+1)} =& \lambda_2^{(j)} + \Phi x_0 + \Psi \zeta^{(j+1)} - x^{(j+1)}
\end{aligned}
\end{equation}
Problem (\ref{eqn_ADMM_problem}) can be shown to be equivalent to the canonical ADMM form {\cite[equation (3.1)]{Wang2014}}, for which the iteration is equivalent to (\ref{ADMM_iteration}). As $F(u) + \Lambda(u,x)$ is convex, it can therefore be concluded that iteration (\ref{ADMM_iteration}) will converge to the solution of (\ref{eqn_ADMM_problem}) since the residuals defined by
\begin{align*}
r_{\text{P}}^{(j+1)} = & \begin{bmatrix}I & 0 \\ 0 & -I \end{bmatrix} \begin{bmatrix}
u^{(j+1)} \\ x^{(j+1)}
\end{bmatrix} + \begin{bmatrix} I \\ \Psi \end{bmatrix} \zeta^{(j+1)} + \begin{bmatrix}
0 \\ \Phi x_0
\end{bmatrix}  \\
 r_{\text{D}}^{(j+1)} =&  \begin{bmatrix} \rho_1 I \\ -\rho_2 \Psi  \end{bmatrix} \begin{bmatrix} \zeta^{(j)} - \zeta^{(j+1)} \end{bmatrix} 
\end{align*}
necessarily converge to zero. The algorithm is terminated when the conditions $\|r_{\text{P}}^{(j+1)}\| \leq \epsilon$ and $\|r_{\text{D}}^{(j+1)}\| \leq \epsilon$ are met, where $\epsilon$ is a pre-determined threshold. 
\begin{algorithm}
\caption{Alternating Direction Method of Multipliers\label{algorithm_ADMM}}
\begin{algorithmic}[1]
\STATE Initialize $u^{(0)}$, $x^{(0)}$, $\zeta^{(0)}$, $\lambda_1^{(0)}$, $\lambda_2^{(0)}$ with (\ref{ADMM_intial_values})
\STATE $j \gets 0$
\WHILE{$\|r_{\text{P}}^{j+1}\| > \epsilon$ and $\|r_{\text{D}}^{j+1}\| > \epsilon$}
\STATE Calculate $u^{(j+1)}$, $x^{(j+1)}$, $\zeta^{(j+1)}$, $\lambda_1^{(j+1)}$, $\lambda_2^{(j+1)}$ from (\ref{ADMM_iteration})
\STATE $j\gets j+1$
\ENDWHILE
\STATE $u^\star \gets u^{(j)}$
\end{algorithmic}
\end{algorithm}

\subsection{Complexity}\label{ADMM_complexity}  
Algorithm \ref{algorithm_ADMM} shows a pseudocode implementation of the ADMM algorithm, and the computational complexity of each recursive variable update is presented in Table \ref{ADMM_complexity}. Each $u_k$ update is an unconstrained convex optimization problem that we solve here using a Newton method with a backtracking line search, so the $u$ update therefore scales linearly with $N$ if these updates are performed sequentially, or is constant if each $k$ update can be performed in parallel. The matrix inversion in the $\zeta$ update can be computed offline as it involves no decision variables, so only a dense matrix multiplication is required. We note that multiplication by $\Psi$ is not considered a matrix multiplication in the analysis presented in Table \ref{ADMM_complexity_table}, as it is the equivalent of a linear filtering operation and therefore scales linearly with $N$. This implies that the residual updates also scale linearly with $N$ if they are analytically block multiplied. The complexity of the ADMM iteration is therefore $\mathcal{O}(N^n)$ where $n \leq 2$ is determined by the method used for matrix multiplication.

\begin{table}
\begin{center}
\caption{ Summary of the matrix/vecotr present in the ADMM variable updates, where 'v' refers to a vector, and 'M' refers to a non-diagonal matrix other than $\Psi$.}\label{ADMM_complexity_table}
\begin{tabular}{c c c c c}
Update & M$^{-1}$ & M$\cdot$M & M$\cdot$v & $\mathcal{O}(N^n)$  \\ \hline
$u$ & No & No & No & $n = 1$ \\
$x$ & No & No & No & $n = 1$ \\
$\zeta$ & No & No & Yes & $n \leq 2$ \\
$\lambda_1, \lambda_2$ & No & No & No & $n = 1$ \\
$r_{\text{P}},r_{\text{D}}$ & No & No & No & $n = 1$ \\
\end{tabular}
\end{center}
\end{table}

\section{Numerical Experiments}\label{section_numerical_experiments}

To compare the performance of the algorithms without reference to a particular PHEV powertrain, single-shot instances of problem (\ref{eqn_convex_problem}) were created with randomly generated parameters. For each instance of the energy manangement problem, a nominal sampling frequency of 1Hz (i.e $\delta=1$\,s) was assumed, and it was also assumed that the engine is always on and the clutch is engaged i.e $\hsk = 1$ (for the purposes of these experiments the switching heuristic is arbitrary). Using observations from previous experiments \cite{Buerger}, predictions were made of driver power demand as $\hpdrvk \in [-2.5 \times 10^3,10\times10^{3}]\,W$ and hardware parameters were generated from the distributions $\ak2,\bk2 \in [0.5\times10^{-5},1.5\times10^{-5}]\,W^{-1}$, $\ak1,\bk1 \in [0.5,1.5]$, and $\ak0,\bk0 = 0\,W$, with $V_{oc}=300\,V$ and $R=0.1\, \Omega$. The limits on state and input were set at $\overline{x} = 10^{5} \,J$, $\underline{x} = 0 \, J$, $\ou = 15\times10^3 \, W$ and $\uu = -15\times10^3 \, W$, and an initial state of charge of $x_0 = 0.9\overline{x}$ was assumed. These limits have little physical significance within the context of this experiment, and were chosen to ensure that the problems are feasible and that both the state and input constraints were active. Finally, the effect of varying $\tau$ was not investigated, and was set at $\tau = 0.995$ for all projected interior point solutions. The simulations were implemented in Matlab on a 2.6GHz Intel Core i7-6700HQ CPU.

\subsection{Optimum}

It is demonstrated in Section \ref{section_interior_point} that the output of Algorithm \ref{algorithm_PPDIPM} can be made arbitrarily close to the solution of (\ref{eqn_ip_problem_1}) by using a sufficiently large value of $\overline{\mu}$, and it is therefore necessary to determine a value of $\overline{\mu}$ that can be used to obtain a sufficiently accurate approximation of $u^\star$. 10 problems for each horizon length $N=100,200,300,400$ were generated for a total of 40 problems, and the projected interior point method was used to obtain a solution for each with the parameters $\mu_0 = \overline{\mu}$, and $\overline{\mu}$ iteratively increased from $10^2$ to $10^5$ in 20 logarithmically spaced points ($k_\mu$ is not required as the algorithm will terminate when condition (\ref{termination_criteria}) is first met). For each $\bar{\mu}_i$, $i=1,\ldots,20$, the control input vector, $u^\ast_i$, at termination was recorded, forming the sequence $u^{\ast}_1,\dots,u^{\ast}_{20}$ for each problem. Fig. \ref{figure1} shows that as $\overline{\mu}$ was increased, the norm of the difference between the values of $u^{(j)}$ when the algorithm terminated with successive values of $\bar{\mu}$, $\|u^{\ast}_i - u^{\ast}_{i-1}\|$, decreased within an inverse band for all cases, and that at $\overline{\mu} = 10^5$ this metric has reduced to less than $1$ for all 40 problem cases. Given that the decision variable, $u$, can take a range of values in the order of $10^4$, it was concluded that $\overline{\mu}=10^5$ is therefore sufficiently large to provide a highly accurate solution to problem (\ref{eqn_ip_problem_1}), and all future references to $u^\star$ refers to control inputs found using Algorithm \ref{algorithm_PPDIPM} with $\mu_0 = \overline{\mu} = 10^5$.

\begin{figure}
\includegraphics[scale=1]{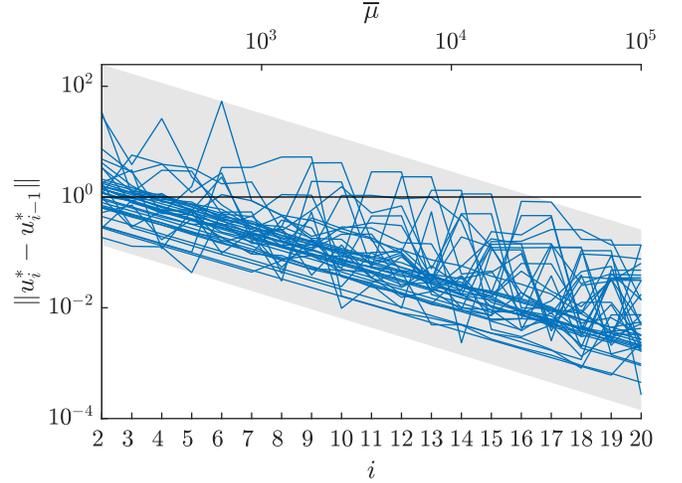}
\caption{Data showing the decrease in change of control vector obtained by the projected interior point method as $\overline{\mu}$ is increased from $10^0$ to $10^5$. The grey shaded area shows the minimum width linear band that contains all of the data points, and the dashed lines are simply used to highlight values on the vertical and horizontal axes. \label{figure1}}
\end{figure}

\subsection{Algorithm Tuning}

Both algorithms have multiple parameters that must be tuned to provide computationally efficient solutions. For the ADMM algorithm, $\rho_1$ and $\rho_2$ (which can be loosely interpreted as the step length in a gradient descent algorithm) must be determined, whilst $\mu_0$ and $k_\mu$ must be determined for the projected interior point method. The energy management MPC framework is commonly implemented with a shrinking horizon, and it is therefore important that the same set of parameters provide a similar level of performance for a broad class of problems over both long and short horizons. This section details the results of investigations to determine the most computationally efficient combination of parameters for each algorithm.

To determine the values of $\rho_1$, $\rho_2$, $\mu_0$, and $k_\mu$, that provide optimal convergence for the ADMM and projected interior point algorithms, 20 new problems were generated for each horizon length of of $N=100,200,300,400$. These were solved using the projected interior point algorithm with $\overline{\mu} = 10^5$, $10^{-5}\leq \mu_0 \leq 10^{1}$, and $1 < k_\mu \leq 10^{5}/\mu_0$ (any value of $k_\mu$ greater than this would ensure that $\mu$ is projected onto $\overline{\mu}$ during the first update step (\ref{mu_update})). For each problem instance, the number of iterations required for the algorithm to terminate were recorded, and the average for each combination of parameters is shown in Fig. \ref{figure2}. It can be seen that there is a vertically banded region at $\mu_0 \approx 10^{-1}$ that requires a minimum number of iterations for all horizon lengths, and that there is a profile to the search space that varies little with changes in horizon length. The values $\mu_0 = 10^{-1}$ and $k_\mu = 10^4$ were therefore selected as the optimal parameters. 

For the ADMM algorithm a different approach was taken, as the number of iterations required to achieve the same level of accuracy as the projected interior point algorithm with $\overline{\mu}=10^5$ made a similar parameter search intractable. Instead, a total of 100 ADMM iterations were completed for each problem (ignoring the stated termination criteria) with $10^{-6} \leq \rho_1 \leq 10^{-2}$ and $10^{-8} \leq \rho_2 \leq 10^{-4}$. The average norm of the difference between the control input at the 100th iteration of ADMM, $u^{(100)}$, and the optimum, $u^\star$, was recorded for each case. The results are shown in Fig. \ref{figure2}, and there is a clear region within approximately two orders of magnitude of both $\rho_1$ and $\rho_2$ where the control vector has a minimum error relative to the optimum, and this region does not change significantly with horizon length. The values of $\rho_1 = 6 \times 10^{-5}$ and $\rho_2 = 4 \times 10^{-7}$ were therefore selected as the optimal parameters.

\begin{figure}
\begin{center}
\includegraphics[scale=1]{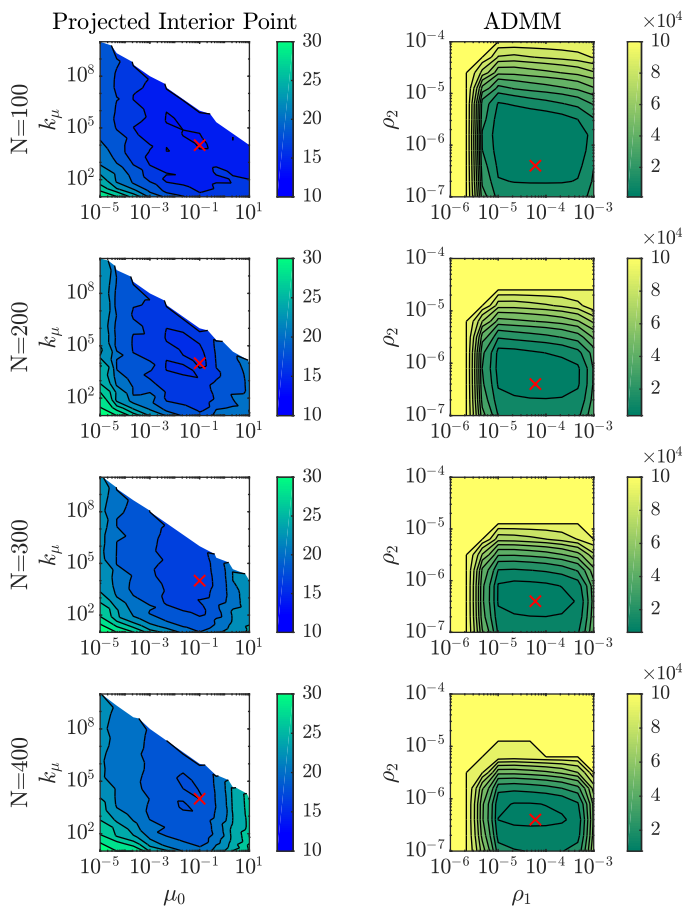}
\caption{Results of parameter tuning for both the projected interior point and ADMM algorithm. The projected interior point figures show the average number of iterations required to meet the termination criteria $\overline{\mu}=10^5$ and are saturated at 30 iterations, whereas the ADMM figures show the average error measured by $\|u^{(100)} - u^\star\|$ and are saturated at $10\times10^4$. The red crosses show the chosen values for simulations described in Section \ref{computational_performance}. \label{figure2}}
\end{center}
\end{figure}

\subsection{Computational Performance}\label{computational_performance}

After tuning the parameters of both the projected interior point and ADMM algorithm to the class of problems being investigated, it was possible to analyse their comparative computational performance. This was achieved in two steps: firstly the termination criteria for a `sufficiently' accurate solution was determined, then the variation in computational time with horizon length was investigated. 

A further 20 test cases were generated consisting of 5 cases for each of $N=100,200,300,400$, and using the values of $\rho_1$, $\rho_2$, $\mu_0$ and $k_\mu$ determined during the tuning phase, each problem was solved using ADMM for 100 iterations, and using the projected interior point algorithm with $\overline{\mu} = 10^5$. The absolute difference between the cost evaluated at iteration $j$ and the optimal cost, $| F(u^{(j)}) - F(u^\star)|$, is shown in Fig. \ref{figure3}. The results clearly demonstrate sublinear convergence for the ADMM algorithm, whilst the projected interior point results show superlinear convergence. Therefore, the projected interior point algorithm can produce an extremely accurate solution within a few tens of iterations, whereas significantly more iterations are required for ADMM. 

\begin{figure}
\includegraphics[scale=1]{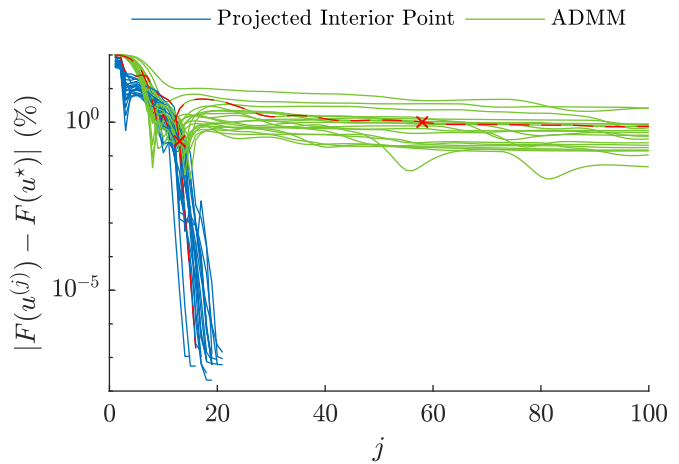}
\caption{Curves showing the normalised error between the cost evaluated at iteration $j$ and the optimum, as a percentage, for 20 systems using both the projected interior point method and ADMM. The curves highlighted in red correspond to the system illustrated in Fig. \ref{figure4}, and the red cross shows the iteration from which those curves were taken. \label{figure3}}
\end{figure}

A threshold of 1\% was considered high enough for `sufficient' accuracy, as this is likely to be lower than the level of uncertainty in state measurements used to formulate the problem, and the results shown in Fig. \ref{figure4} illustrate that the deviation between the control vectors obtained by both the ADMM and projected interior point algorithms and the optimum are almost imperceptible at this level of convergence. A slightly larger deviation is observed between the state trajectories, particularly that obtained with ADMM, however this is because the state trajectory is a function of the integral of the control input, and as the cost is not a function of state-of-charge this does not necessarily indicate greater sub-optimality. A key property of the algorithms is also demonstrated in Fig. \ref{figure4}: the state constraints are only guaranteed for both algorithms when the residuals, $r_{\text{IP}}, r_{\text{P}},$ and $r_{\text{D}}$, are precisely zero (this is also why we use the absolute error in Fig. \ref{figure3}, as the cost evaluated for each iteration of can be \textit{lower} than $F(u^\star)$). Therefore, the termination criteria do not provide a guarantee of enforcing the state constraints, and we can see that for the final three timesteps the lower state limit is violated by $\approx 2.5\%$ of the feasible state band for the ADMM trajectory. This limitation can be reduced by tightening the algorithms' convergence thresholds, which makes it more significant for ADMM due to its sublinear rate of convergence.

\begin{figure}
\includegraphics[scale=1]{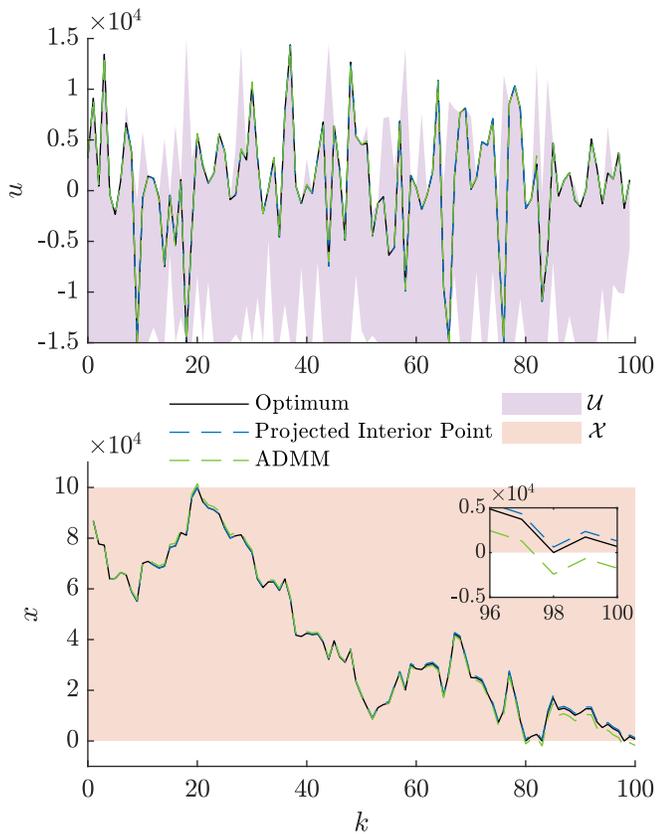}
\caption{The control and state vectors for the systems highlighted with red crosses in Fig. \ref{figure3}, plotted against the optimum $u^\star$ and $x^\star$. The feasible tubes $\U$ and $\X$ are also included, and note that the additional constraints enforced during the convex formulation specified in section \ref{section_convex_reformulation} have significantly restricted the upper and lower bounds on $\U$ from the original $\pm 1.5\times10^4$. \label{figure4}}
\end{figure}

Based on residuals for the projected interior point and ADMM trajectories shown in Fig. \ref{figure4}, it was assumed that $\epsilon = 4 \times 10^3$ and $\overline{\mu} = 1 $ enforce a `sufficient' level of convergence. A further 20 problems were generated for horizons $50 \leq N \leq 1000$, and the iterations to completion, mean time taken per iteration, and time to completion were recorded for each using both ADMM and the projected interior point algorithm. For comparison, the problems were also solved using CVX with default solver SDPT3 v.4.0 \cite{Toh1999} and default error tolerance, for which only the total time was recorded (it is not possible to separate the total time from the individual iterations or the overhead required to parse the problem when using CVX).  The results are shown in Fig. \ref{figure5}, where it can be seen that whilst the uncertainty in the number of ADMM iterations is high (from as low as 50 to as high as 400), the band of uncertainty is near constant as the horizon is increased, so it can be assumed that the expected number of iterations is effectively constant with horizon length. The uncertainty for the number of projected interior point iterations is lower, and fewer iterations are required for all horizon lengths, however the number of iterations also increases linearly with horizon length from $\sim10$ iterations at $N=50$ to $\sim16$ iterations at $N=1000$.

\begin{figure}
\includegraphics[scale=1]{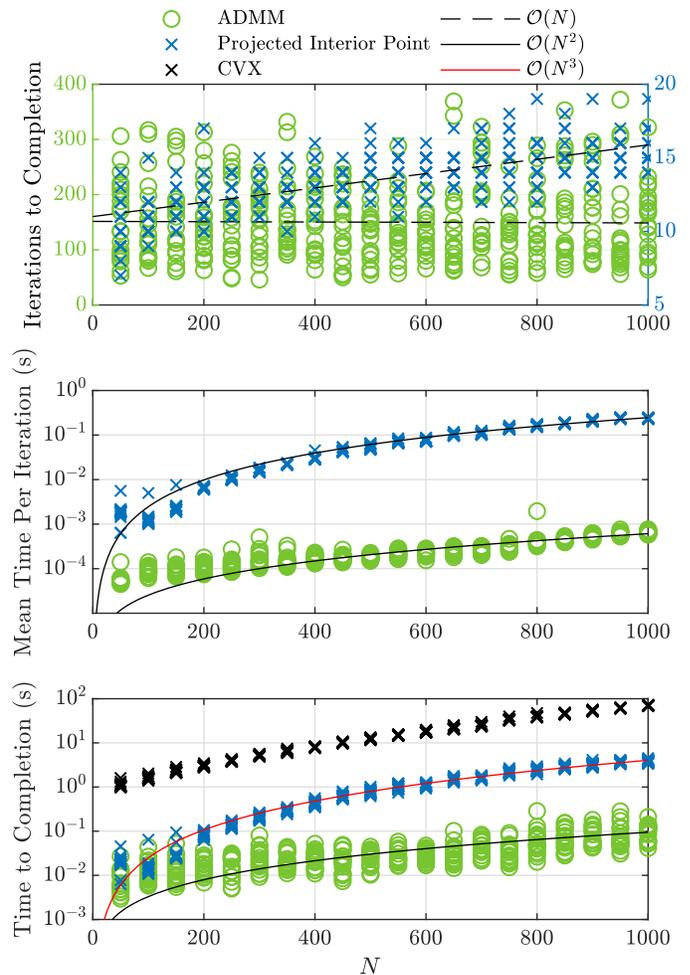}
\caption{Results showing the number of iterations required, mean time per iteration, and time to completion for 20 systems with $50 \leq N \leq 1000$, with linear, quadratic, and cubic trendlines. \label{figure5}}
\end{figure}

It was shown in sections \ref{IP_complexity} and \ref{ADMM_complexity} that, as the horizon length is increased, the computational burden of the projected interior point algorithm is dominated by the $\Delta \theta_1$ update and the ADMM algorithm is dominated by the $\zeta$ update, so the methods used to perform these calculations will largely determine the time required per iteration. The Matlab operations \texttt{x=A\textbackslash b} and \texttt{x=A$^*$b} were used here, and a quadratic trendline is shown to have an approximate fit for both in Fig. \ref{figure5}, although as $N$ was increased towards $1000$, the projected interior point iterations took over two orders of magnitude longer than the ADMM iterations. It was therefore expected that the total time taken for the ADMM algorithm would scale quadratically with horizon length, whereas the time taken for the projected interior point method would scale cubically with horizon length, and this is supported by the results shown in the bottom plot in Fig. \ref{figure5}. It is also shown that (assuming an interval of 1 second between controller optimizations) the projected interior point is only suitable up to a horizon of $N \approx 500$, whereas even up to the maximum horizon length of $N=1000$, the ADMM algorithm only required $\sim 0.1s$. From the previous scaling properties it can be assumed that ADMM is real time implementable for horizons significantly in excess of $1000$ (the performance of the ADMM algorithm presented here exceeds that presented in \cite{EastCDC2018} due to a fully vectorized software implementation). Therefore, whilst the projected interior point algorithm has been shown to converge to an extremely accurate solution in fewer iterations than the ADMM algorithm, for the hardware used in these experiments, less time is required for a moderate level of accuracy using ADMM, and the ADMM algorithm scales better with horizon length. If the accuracy requirement were tightened, however, it is likely that this performance relationship would change significantly.

In comparison, CVX was unable to obtain solutions in less than 1s for any horizon length, and was at least an order of magnitude slower than both algorithms over all horizon lengths; compared to ADMM it was a factor of 1000 slower for $N = 1000$. Although CVX is solving the problem to a different error tolerance, it would be expected that ADMM would still be faster were its termination threshold significantly tightened. This is the first demonstration of a method capable of solving the energy management problem in real time, over long horizons ($\geq$1000 samples) when nonlinear system dynamics are considered and hard limits on both power and state of charge are enforced over the entire horizon.

To conclude the numerical experiments, we note that the Newton method for the projected interior point requires the solution to non-diagonal linear systems of equations, which in turn is typically solved using BLAS \cite{BLAS}. This is not an issue when solving the problems on desktop hardware as demonstrated here, but may not be an option for the embedded hardware used for an online solution in a vehicle. In this case only the ADMM algorithm is suitable, as although it also requires a Newton method for the individual control variable updates, this can be performed element-wise and therefore does not require a matrix inversion step.

\section{Conclusion}\label{section_conclusion}
This paper proposes a projected interior point method for the solution of a convex formulation of the optimization problem associated with nonlinear MPC for energy management in hybrid electric vehicles. The performance w.r.t the tailored ADMM algorithm of \cite{EastCDC2018} is demonstrated through numerical experiments, and the projected interior point algorithm is shown to have faster convergence (superlinear) for the class of problems investigated, although the ADMM algorithm is shown to have superior numerical performance and scaling properties when a modest level of accuracy is required. Both algorithms are also shown to have superior computational performance to general purpose convex optimization software.

\bibliographystyle{IEEEtran}
\bibliography{bibliography}





\end{document}